\documentclass{article}
\usepackage{authblk}
\usepackage{amsmath}
\usepackage{amssymb}
\usepackage{amsthm}
\usepackage{thmtools}

\newtheorem{theorem}{Theorem}[section]
\newtheorem{lemma}[theorem]{Lemma}
\newtheorem{proposition}[theorem]{Proposition}
\newtheorem{definition}[theorem]{Definition}
\newtheorem{corollary}[theorem]{Corollary}

\newtheorem{question}[theorem]{Question}
\newtheorem*{theorem*}{Theorem}

\DeclareMathOperator{\ord}{ord}

\newcommand{\Q}{\mathbb Q}
\newcommand{\dd}{\delta}
\newcommand{\pa}{\partial}
\newcommand{\Dt}{\Delta}

\newcommand{\GL}{{\hbox{\rm GL}}}

\newcommand{\lclm}{{\rm lclm}}
\newcommand{\gcrd}{{\rm gcrd}}

\newcommand{\bfm} { {\mathbf{m}}}

\newcommand{\bfi} { {\mathbf{i}}}
\newcommand{\bfj} { {\mathbf{j}}}

\newcommand{\calD}{{\mathcal{D}}}

\newcommand{\calN}{{\mathcal{N}}}
\newcommand{\calL}{{\mathcal{L}}}
\newcommand{\calT}{{\mathcal{T}}}
\newcommand{\calU}{{\mathcal{U}}}

\title{The Equivalence Problem for Generalized Airy Operators
\thanks{This work was supported by the Strategic Priority Research Program of Chinese Academy of Sciences under Grant XDA0480503 and XDA0480502, 
the NSFC Grants No. 12471479}}
\author[1,2]{Yichuan Cao}
\author[1,2]{Ruyong Feng}
\author[1,2]{Yunfei Li}
\author[2]{Ruichen Qiu}
\affil[1]{State Key Laboratory of Mathematical Sciences,
Academy of Mathematics and Systems Science, Chinese Academy of Sciences}
\affil[2]{
School of Mathematics, University of Chinese Academy of Sciences}
\date{}

\begin{document}
\maketitle

\begin{abstract}
In this paper, we establish degree obstructions to the equivalence of generalized Airy operators of the same type. As an application, we answer a question posed by Nicholas M. Katz in Inventiones Mathematicae (87, pp. 13–61,1987). The main results of Sections 3 and 4 were obtained through a close interactive collaboration between the authors and the artificial intelligence agent system MechMath Agent Team (MMAT).
\end{abstract}

\section{Introduction}
Let $k$ be an algebraically closed field of characteristic zero, and let $k(x)$ denote the field of rational functions over $k$, equipped with the usual derivation
$
\delta=\frac{d}{dx}.
$
Let
$
R=k(x)[\partial]
$
be the ring of differential operators over $k(x)$, where the multiplication is determined by
\[
\partial f=f\partial+\delta(f), \qquad f\in k(x).
\]
The ring $R$ is a typical example of an Ore ring. We refer the reader to \cite{Ore:theoryofnoncommutativepolynomials} for some of its basic properties.
\begin{definition}
For two nonzero operators $\calL_1,\calL_2\in R$, we say $\calL_1$ is equivalent over $k(x)$ to $\calL_2$, 
denoted by $\calL_1\sim \calL_2$, if there exists a nonzero $\calT \in R$ such that 
\begin{equation}
  \label{EQ:equivalence}
   \gcrd(\calT,\calL_1)=1,\,\, \lclm(\calT, \calL_1)=\calL_2\calT,
\end{equation} 
where $\gcrd$ denotes the greatest common right-hand divisor, and $\lclm$ denotes the least common left-hand multiple.
In this situation, we say that $\calL_1$ is transformed into $\calL_2$ by $\calT$.
\end{definition}
\begin{definition}\label{def:airy_operator}
A generalized Airy operator of type \((n,m)\) is an operator
\[
  \calL=P(\partial)+Q(x),
\]
where \(P,Q\in k[x]\), \(P\) is monic of degree \(n\), \(Q\) is of degree \(m\), and \(Q(0)=0\).
\end{definition}
The generalized Airy operators were systematically studied by Katz in 
\cite{Katz:calculationofdifferentialGaloisgroups,Katz:ExponentialSumsAndDifferentialEquations}. 
In particular, Katz showed in \cite{Katz:calculationofdifferentialGaloisgroups} that generalized Airy operators of type $(n,m)$ 
with $\gcd(n,m)=1$ have large differential Galois groups. Furthermore, he proved that, when $\gcd(n,m)=1$, 
the differential Galois group of a generalized Airy operator is completely determined by 
whether the coefficient of the second highest-order term of $P$ vanishes and whether the operator is self-dual (see Theorem 4.2.7 of 
\cite{Katz:calculationofdifferentialGaloisgroups}). 
To determine whether a generalized Airy operator is self-dual, he posed the following question. 
\begin{question}
\label{Ques:Katz}
Suppose that $\calL_1,\calL_2$ are two generalized Airy operators of type $(n,m)$ with $\gcd(n,m)=1$. 
If $\calL_1\sim \calL_2$, then $\calL_1=\calL_2$.
\end{question}
Following Katz, Saidane in \cite{Saidane:airyoperatorsofsmallorder} studied reducibility, equivalence, and self-duality for generalized
Airy operators of small order, obtaining explicit criteria in several cases of order at
most three. In particular, he provided an affirmative answer of Katz's problem in the case of order three and proved that no generalized Airy
operator of type $(3, m)$ with $m>1$ is self-dual.

In this paper, we shall give an affirmative answer to Katz's problem. 
The remainder of this paper is organized as follows. 
In Section 2, we investigate the formal solutions of generalized Airy operators. 
In Section 3, we derive degree-theoretic necessary conditions for two generalized
Airy operators to be equivalent. Section 4 is devoted to the proofs of the main results.

\section{Formal solutions of generalized Airy operators}
Let $\calU$ be a universal differential extension of $k(x)$. We regard $\calU$ as a left $R$-module by defining
$$
\calL(\xi)=a_n\delta^n(\xi)+\cdots+a_1\delta(\xi)+a_0\xi,
$$
for
$
\calL=a_n\partial^n+\cdots+a_1\partial+a_0\in R.
$
and $\xi\in \calU$.
An element $\xi\in\calU$ is called a solution of $\calL$ if
$
\calL(\xi)=0.
$
The space of solutions of a nonzero operator $\calL$ has dimension $\ord(\calL)$ over the field of constants of $\calU$.
In this section, we investigate the formal solutions of generalized Airy operators at the point $x=\infty$. 

Write $t=1/x$. For each positive integer $r$, let $k((t^{1/r}))$ denote the field of Laurent series in $t^{1/r}$ 
over $k$. By the Newton--Puiseux theorem, the field
\[
  \bigcup_{r\geq 1} k((t^{1/r}))
\]
is an algebraic closure of $k((t))$.
We denote by $\ord_t$ the valuation on $\bigcup_{r=1}^\infty k((t^{1/r}))$, defined by
$$
\ord_t(\xi)=s,
$$
where
$
\xi=a_s t^s+\text{higher order terms},
$
with $a_s\neq 0$ and $s\in\mathbb{Q}$. By convention, we set $\ord_t(0)=\infty$.  
\begin{proposition}
\label{PROP:Puiseuxsolutions}
Let $P\in k[Y]$ be monic of degree $n>0$, and let $Q\in k[x]$ have degree $m>0$ and leading coefficient $b_m$. Put
\[
  e=\frac{n}{\gcd(n,m)},\qquad q=\frac{m}{\gcd(n,m)},\qquad v=t^{1/e}.
\]
If $c_1,\dots,c_n$ are the roots of $W^n+b_m=0$, then there exist unique units $w_{c_i}(v)\in k[[v]]$ with $w_{c_i}(0)=c_i$ such that
\[
  v^{-q}w_{c_1}(v),\dots,v^{-q}w_{c_n}(v)
\]
are precisely the roots of $P(Y)+Q(x)=0$ in $\bigcup_{r=1}^\infty k((t^{1/r}))$.
\end{proposition}
\begin{proof}
Write $Z=Y^{-1}$.  Then
$
P(Z^{-1})=Z^{-n}A(Z),
$
where
$$
A(Z)=1+a_{n-1}Z+\cdots+a_0Z^n.
$$
As $t=1/x$, we have that 
$
Q(t^{-1})=b_mt^{-m}B(t),
$ 
where
$$
B(t)=1+\frac{b_{m-1}}{b_m}t+\cdots+\frac{b_0}{b_m}t^m.
$$
Thus \(A(0)=B(0)=1\). The equation
$
P(Y)+Q(x)=0
$
becomes
\[
Z^{-n}A(Z)+b_mt^{-m}B(t)=0,
\]
or equivalently
\begin{equation}\label{EQ:inverse}
 Z^n=-b_m^{-1}t^m\frac{A(Z)}{B(t)}.
\end{equation}
Since all roots of $P(Y)+Q(x)=0$ are nonzero, $\eta$ is a solution of $P(Y)+Q(x)=0$ if and only 
if $\eta^{-1}$ is a root of (\ref{EQ:inverse}).
We seek a solution of (\ref{EQ:inverse}) in the form $Z=v^qw$. 
Substituting $Z=v^qw$ into \eqref{EQ:inverse}, and using
\(qn=me\), gives
\[
v^{qn}w^n
 =-b_m^{-1}v^{em}\frac{A(v^qw)}{B(v^e)}.
\]
Since \(qn=em\), the powers of \(v\) cancel, and we obtain
$
 w^n=-b_m^{-1}\frac{A(v^qw)}{B(v^e)}.
$
Set
\[
F(v,W)
 =W^n+b_m^{-1}\frac{A(v^qW)}{B(v^e)}.
\]
Then $w$ is a solution of $F(v,W)=0$ in $\bigcup_{r=1}^\infty k((t^{1/r}))$. At \(v=0\), we have
\[
F(0,W)=W^n+b_m^{-1}.
\]
Let $\alpha_1,\dots,\alpha_n\in k$ be all solutions of $F(0, W)=0$. Since $\alpha_i\neq 0$, one has that
\[
\frac{\partial F}{\partial W}(0,\alpha_i)=n\alpha_i^{n-1}\neq 0. 
\]
Therefore, by the formal implicit function theorem, for each $\alpha_i$, there exists a unique series
\[
\tilde{w}_{\alpha_i}(v)\in k[[v]],
\qquad \tilde{w}_{\alpha_i}(0)=\alpha_i,
\]
satisfying $F(v, \tilde{w}_{\alpha_i}(v))=0$. In particular, $\tilde{w}_{\alpha_i}(v)$ is a unit in
$k[[v]]$. Hence
$
v^q\tilde{w}_{\alpha_i}(v)
$
is a solution of $F(v,W)=0$ in $k((v))$ and thus $v^{-q}\tilde{w}_{\alpha_i}(v)^{-1}$ is a solution of $P(Y)+Q(x)=0$ in $k((v))$. Set $c_i=\alpha^{-1}$ and $w_{c_i}=\tilde{w}_{\alpha_i}(v)^{-1}$ for each $1\leq i \leq n$. Then $w_{c_i}(v)\in k[[v]]$ and $w_{c_i}(0)=\alpha_i^{-1}$ which is a root of $W^n+b_m=0$.  
Finally, since $c_1,\dots,c_n$ are distinct, $v^{-q}w_{c_1}(v),\dots, v^{-q}w_{c_n}(v)$ are distinct and thus they form all roots of $P(Y)+Q(x)=0$.
\end{proof}

In the following, for $\eta \in \calU$, the expression \(\exp\left(\int \eta\,dx\right)\) denotes a nonzero \(E\in \calU\) satisfying
$
  \dd(E)=\eta E.
$
By convention, we also write $\calU[\partial]$ for the ring of differential operators over $\calU$. 
\begin{lemma}
\label{LM:solution-reduced}
Let $\calL=P(\partial)+Q(x)$ be a generalized Airy operator, $\eta, h\in \calU$ and $E=\exp^{\int \eta dx}$.  Then
\[
  E^{-1}\calL(Eh)=(P(\partial+\eta)+Q(x))(h).
\]
\end{lemma}

\begin{proof}
 Let \(\calD=\partial+\eta\). For $a\in \calU$, we have 
\[
  \partial(Ea)=\dd(Ea)=\dd(E)a+E\dd(a)=\eta Ea+E\dd(a)=E(\eta a+\dd(a))=E\calD(a).
\]
Thus, we obtain
$
  \partial E = E \calD,
$
where $E$ is identified with the multiplication operator by $E$.
It follows by induction on \(j\) that
$
  \partial^jE=E \calD^j
$
and thus $P(\partial)E=EP(\calD)$.  Since $\calL=P(\partial)+Q(x)$, we obtain
\begin{align*}
  E^{-1}\calL(Eh)&=E^{-1}((P(\partial)E)(h)+Q(x)Eh)=E^{-1}(EP(\calD)(h)+EQ(x)h)\\
  &=(P(\calD)+Q(x))(h).
\end{align*}
\end{proof}

\begin{lemma}
\label{LM:transport_recursion}
Let $\calL=P(\partial)+Q(x)$ be a generalized Airy operator of type $(n,m)$ with $mn\neq 0$, 
and let $\eta$ be a solution of $P(Y)+Q(x)=0$ in $\bigcup_{r=1}^\infty k((t^{1/r}))$. Set $e=n/\gcd(n,m)$ and $\tau=(n-1)m/(2n)$. Then there exists a nonzero $h\in k((t^{1/(2e)}))$ with $\ord_t(h)=\tau$ such that 
$$
    (P(\partial + \eta) +Q(x))(h)=0.
$$
\end{lemma}

\begin{proof}
We first determine the order of $(\partial+\eta)^k(t^\rho)$. Expanding $(\partial+\eta)^k$, each monomial is of the form
$$
\bfm_{\bfi,\bfj}=\eta^{i_1}\partial^{j_1}\eta^{i_2}\partial^{j_2}\cdots \eta^{i_s}\partial^{j_s},
$$
where $i_\ell,j_\ell$ are nonnegative integers satisfying
$$
i_1+\cdots+i_s+j_1+\cdots+j_s=k.
$$
Since $P(\eta)+Q(x)=0$, the terms involving only powers of $\eta$ cancel out. Therefore, it suffices to consider monomials for which at least one of $j_1,\dots,j_s$ is nonzero.
As $\ord_t(\eta)=-m/n$ and $\delta(t^\rho)=-\rho t^{\rho+1}$, one sees that if 
$\bfm_{\bfi,\bfj}(t^\rho)\neq 0$
 then
\begin{equation}
  \label{EQ:ordmnomials}
  \ord_t (\bfm_{\bfi,\bfj}(t^\rho))=\rho-\left(i_1+\dots+i_s\right)\frac{m}{n}+(j_1+\dots+j_s).
\end{equation}
Therefore, 
$\ord_t(\bfm_{\bfi,\bfj}(t^\rho))\geq \rho-(k-1)m/n+1$
 for all monomials in $(\partial+\eta)^k-\eta^k$. Thus for the monomials $\bfm_{\bfi,\bfj}$ in $P(\partial+\eta)+Q(x)$, one has that 
 $$
     \ord_t(\bfm_{\bfi,\bfj}(t^\rho))\geq \rho-\frac{(n-1)m}{n}+1
 $$
 and the equality holds only if the monomials are of the form 
$$
     \eta^{n-1-i}\partial \eta^{i},\,\,0\leq i \leq n-1.
$$ 
Set 
$$
 \calN_1=\sum_{i=0}^{n-1} \eta^{n-1-i}\partial \eta^{i} \,\,\text{and}\,\, \calN_2=P(\partial+\eta)+Q(x)-\calN_1.
$$
Then $P(\partial+\eta) + Q(x)=\calN_1+\calN_2$. 
Assume that $\bfm_{\bfi,\bfj}$ is a monomial appearing in $\calN_2$.  If $\sum_{\ell=1}^s(i_\ell+j_\ell)=n$ then $\sum_{\ell=1}^s j_\ell\geq 2$ and from (\ref{EQ:ordmnomials}), we have
$$
     \ord_t(\bfm_{\bfi,\bfj}(t^\rho))\geq \rho-\frac{(n-2)m}{n}+2>\rho-\frac{(n-1)m}{n}+1+\frac{m}{n}.
$$
If $\sum_{\ell=1}^s(i_\ell+j_\ell)\leq n-1$, then $\sum_{\ell=1}^s j_\ell\geq 1$ and from (\ref{EQ:ordmnomials}) again, we have
\begin{align*}
  \ord_t(\bfm_{\bfi,\bfj}(t^\rho))&\geq \rho-\frac{(\sum_{\ell=1}^s i_\ell)m}{n}+\left(\sum_{\ell=1}^s j_\ell\right)\\
  &\geq \rho - \frac{(n-2)m}{n}+1 = \rho-\frac{(n-1)m}{n}+1+\frac{m}{n}.
\end{align*}
In summary, we have that if $\bfm_{\bfi,\bfj}(t^\rho)\neq 0$ then
\begin{equation}
\label{EQ:minimalorder}
   \ord_t(\bfm_{\bfi,\bfj}(t^\rho))=\begin{cases} 
  \rho-\frac{(n-1)m}{n}+1 & \text{$\bfm_{\bfi,\bfj}$ in $\calN_1$}\\
   \geq \rho-\frac{(n-1)m}{n}+1+\frac{m}{n} & \text{$\bfm_{\bfi,\bfj}$ in $\calN_2$}.

      \end{cases}
\end{equation}
Now let us compute the leading coefficient of $(\calN_1+\calN_2)(t^\rho)$. We first investigate the leading coefficient of $\bfm_{\bfi,\bfj}(t^\rho)$ with minimal order.  
By Proposition~\ref{PROP:Puiseuxsolutions}, 
$$\eta=c_0t^{-\frac{m}{n}}+c_1t^{-\frac{m}{n}+\frac{1}{e}}+\text{higher order terms}$$ 
for some nonzero $c_0\in k$. One has that for the monomials $ \eta^{n-1-i}\partial \eta^{i}$,
\begin{align*}
   (\eta^{n-1-i}\partial \eta^{i})(t^\rho)=\eta^{n-1-i}\delta(\eta^{i}t^\rho)&=c_0^{n-1}\left(\rho-\frac{i m}{n}\right)t^{\rho-\frac{(n-1)m}{n}+1}\\ \notag
&+a_{i}t^{\rho-\frac{(n-1)m}{n}+\frac{1}{e}+1}+\text{higher order terms}.
\end{align*}
Summing over $i$ together with (\ref{EQ:minimalorder}) gives 
\begin{align}
\label{EQ:leadingcoefficient}
    (\calN_1+\calN_2)(t^\rho)&=c_0^{n-1}\left(n\rho-\frac{(n-1)m}{2}\right)t^{\rho-\frac{(n-1)m}{n}+1}\\ \notag
&+\tilde{a}t^{\rho-\frac{(n-1)m}{n}+\frac{1}{e}+1}+\text{higher order terms}.
\end{align}
Let $h_0=t^{\tau}$. Then by (\ref{EQ:leadingcoefficient}) with $\rho=\tau$, we have 
$$
  \ord_t((\calN_1+\calN_2)(h_0))\geq -\tau+\frac{1}{e}+1.
$$
We seek a solution 
$$
   h=t^\tau(1+b_1t^{\frac{1}{e}}+b_2t^{\frac{2}{e}}+\dots)
$$
step by step. Suppose that we already have
$$
    h_i=t^{\tau}(1+b_1t^{\frac{1}{e}}+\dots+b_it^{\frac{i}{e}})
$$
satisfying
$$
  \ord_t((\calN_1+\calN_2)(h_i))\geq -\tau+\frac{i+1}{e}+1.
$$
We shall determine $b_{i+1}$. Write 
$$
    (\calN_1+\calN_2)(h_{i})=\alpha_{i}t^{-\tau+\frac{i+1}{e}+1}+\alpha_{i+1} t^{-\tau+\frac{i+2}{e}+1}+\text{higher order terms}.
$$
We have that  
\begin{align*}
  (\calN_1+\calN_2)&(h_i+b_{i+1}t^{\tau+\frac{i+1}{e}})=(\calN_1+\calN_2)(b_{i+1}t^{\tau+\frac{i+1}{e}})+(\calN_1+\calN_2)(h_i)\\
&=b_{i+1}c_0^{n-1}\left(\frac{n(i+1)}{e}\right)t^{-\tau+\frac{i+1}{e}+1}
+\tilde{b}t^{-\tau+\frac{i+2}{e}+1}+\text{higer order terms}\\
&+\alpha_{i}t^{-\tau+\frac{i+1}{e}+1}+\alpha_{i+1} t^{-\tau+\frac{i+2}{e}+1}+\text{higher order terms}.
\end{align*}
Set 
$$
 b_{i+1}=-\frac{\alpha_{i}e}{n(i+1) c_0^{n-1}}.
$$
Then 
$$
   \ord_t((\calN_1+\calN_2)(h_{i+1}))\geq -\tau+\frac{i+2}{e}+1.
$$
Therefore, $ \ord_t((\calN_1+\calN_2)(h))\geq \infty$ and so $(\calN_1+\calN_2)(h)=0$.
\end{proof}

\begin{lemma}
\label{LM:linearlyindependent}
Let $\calL=P(\partial)+Q(x)$ be a generalized Airy operator of type $(n,m)$ with $nm\neq 0$. Set $e=n/\gcd(n,m)$. Let $\eta_1,\dots,\eta_n\in k((t^{1/e}))$ be all roots of $P(Y)+Q(x)=0$ and $h_i\in k((t^{1/(2e)}))$ be a nonzero solution of $(P(\partial+\eta_i)+Q(x))(y)=0$. Then $\exp^{\int \eta_1 dx}h_1,\dots, \exp^{\int \eta_n dx}h_n$ are linearly independent over $k$.
\end{lemma}
\begin{proof}
Set $z_i=\exp^{\int \eta_i dx} h_i$ for all $1\leq i \leq n$ and write $v=t^{1/e}$. 
By Proposition~\ref{PROP:Puiseuxsolutions}, the roots of $P(Y)+Q(x)=0$ are of the form
\[
\eta_i=v^{-q}w_{c_i}(v),
\]
where $q=m/\gcd(n,m)$,$w_{c_i}(v)\in k[[v]]$ and $w_{c_i}(0)=c_i$, $c_1,\dots,c_n$ are the $n$ distinct roots of 
$
W^n+b_m=0
$
and $b_m$ is the leading coefficient of $Q(x)$.
Consequently,
\[
\eta_i=c_i t^{-m/n}+\text{higher order terms}.
\]

Set
\[
\alpha_i=\frac{\delta (z_i)}{z_i}
=\eta_i+\frac{\delta (h_i)}{h_i}.
\]
For every nonzero Puiseux Laurent series $h_i$, one has
\[
\ord_t\left(\frac{\delta (h_i)}{h_i}\right)\geq 1,
\]
because $\delta=d/dx=-t^2d/dt$. Since $\ord_t(\eta_i)=-m/n<0$, it
follows that
\begin{equation}\label{eq:alpha-leading}
\alpha_i=c_i t^{-m/n}+	\text{higher order terms}.
\end{equation}
For $r\geq 0$, define
\[
B_{r,i}=\frac{\delta^r (z_i)}{z_i}.
\]
We claim that
\begin{equation}\label{eq:Bell-leading}
B_{r,i}=c_i^r t^{-rm/n}
+\text{higher order terms}.
\end{equation}
Indeed, $B_{0,i}=1$, and the logarithmic-derivative recurrence gives
\[
B_{r+1,i}=\delta(B_{r,i})+\alpha_i B_{r,i}.
\]
Assuming \eqref{eq:Bell-leading} for $r$, differentiation raises the
$t$-order of every nonzero Puiseux monomial by $1$, whereas multiplication
by $\alpha_i$ lowers the order by $m/n$. Hence the unique lowest-order
term in $B_{r+1,i}$ comes from the product of the leading terms of
$\alpha_i$ and $B_{r,i}$, and is
\[
c_i^{r+1}t^{-(r+1)m/n}.
\]
This proves \eqref{eq:Bell-leading} by induction.

Now consider the Wronskian
\[
W(z_1,\ldots,z_n)
=\det\bigl(\delta^{r-1}(z_i)\bigr)_{1\leq r,i\leq n}.
\]
Factoring $z_i$ from the $i$th column, we obtain
\[
W(z_1,\ldots,z_n)
=\left(\prod_{i=1}^n z_i\right)
\det\bigl(B_{r-1,i}\bigr)_{1\leq r,i\leq n}.
\]
By \eqref{eq:Bell-leading}, after factoring
$t^{-(r-1)m/n}$ from the $r$th row, the remaining matrix has constant
term
\[
\begin{pmatrix}
1&1&\cdots&1\\
c_1&c_2&\cdots&c_n\\
\vdots&\vdots&&\vdots\\
c_1^{n-1}&c_2^{n-1}&\cdots&c_n^{n-1}
\end{pmatrix}.
\]
Its determinant is the Vandermonde determinant
\[
\prod_{1\leq i<j\leq n}(c_j-c_i),
\]
which is nonzero because the $c_i$ are pairwise distinct. Therefore
\[
W(z_1,\ldots,z_n)\neq 0.
\]
Since the constant field is $k$, the nonvanishing of the Wronskian implies
that $z_1,\ldots,z_n$ are linearly independent over $k$.
\end{proof}

\begin{theorem}
\label{THM:basisofsolutions}
Let $\calL=P(\partial)+Q(x)$ be a generalized Airy operator of type $(n,m)$ with $nm\neq 0$. Set $e=n/\gcd(n,m)$. Let $\eta_1,\dots,\eta_n\in k((t^{1/e}))$ be all roots of $P(Y)+Q(x)=0$ and $h_i\in k((t^{1/(2e)}))$ be a nonzero solution of $(P(\partial+\eta_i)+Q(x))(y)=0$. Then $\exp^{\int \eta_1 dx}h_1,\dots, \exp^{\int \eta_n dx}h_n$ form a basis of the solution space of $\calL(y)=0$.
\end{theorem}

\begin{proof}
By Proposition~\ref{PROP:Puiseuxsolutions} and Lemmas~\ref{LM:solution-reduced} and ~\ref{LM:transport_recursion}
, $\exp^{\int \eta_i dx} h_i$ is a solution of $\calL(y)=0$. The theorem then follows 
from Lemma~\ref{LM:linearlyindependent}.
\end{proof}

\section{Degree obstructions}
In this section, we derive necessary conditions for two generalized Airy operators
of the same type to be equivalent.
\begin{lemma}
\label{LM:logderivative}
Let $w$ be a nonconstant Puiseux Laurent series.  Then 
$$
 \ord_t\!\left(\frac{\dd (w)}{w}\right)\ge 1.
$$ 
Moreover, 
$\ord_t\!\left(\frac{\dd w}{w}\right)=1$ if and only if $\ord_t(w)\neq 0$.

\end{lemma}

\begin{proof}
Write
\[
  w=ct^\lambda(1+U),
\]
with $c\in k\setminus \{0\}, \lambda\in\Q$, and either $U=0$ or
$\ord_t(U)>0$.  Since \(\dd=-t^2d/dt\),
\[
  \frac{\dd (w)}{w}
  =
  -\lambda t+\frac{\dd (U)}{1+U}.
\]
The factor $1+U$ has order $0$.  If $U\neq 0$ and
$\rho=\ord_t(U)>0$, then 
$\ord_t(\dd (U))=\rho+1>1$; for $U=0$, $\ord_t(U)=\infty$.  Thus, if 
$\lambda\neq 0$, the term $-\lambda t$ cannot cancel and
$\ord_t(\dd(w)/w)=1$.  If $\lambda=0$ and $w\notin k$, then
$U\neq 0$, and $\ord_t(\dd(w)/w)>1$.  Because $\ord_t(w)=\lambda$, the
equivalence follows.
\end{proof}

\begin{lemma}
  \label{LM:order-polynomialdifference}
Let $P(Y) \in k[Y]$ be a monic polynomial of degree $n$, and $w, \Dt \in k((t^{1/e}))$ be two Puiseux series. 
If $\ord_t(w) < 0$ and $\ord_t(\Dt) > \ord_t(w)$, then  
$$P(w+\Dt)-P(w) = nw^{n-1}\Dt + M,$$
where $\ord_t(M)>(n-1)\ord_t(w)+\ord_t(\Delta)$.
\end{lemma}
\begin{proof}
Let $P = \sum_{j=0}^n p_j Y^j$. We have 
\[
  P(w+\Dt)-P(w)
  =\sum_{j=0}^np_{j}\sum_{k=1}^j\binom jkw^{j-k}\Dt^k.
\]
The term $nw^{n-1}\Dt$ has order of $(n-1)\ord_t(w)+\ord_t(\Delta)$.
While, the term $w^{j-k}\Dt^k$ with $j<n$ or $k>1$ has order of 
$$ 
 (j-k)\ord_t(w) + k\ord_t(\Dt) = j \ord_t(w) + k(\ord_t(\Dt)-\ord_t(w))>\ord_t(w^{n-1}\Dt).
$$  
\end{proof}

\begin{definition}
  \label{DEF:contact_order}
For
\[
  z=h\exp\left(\int \eta\,dx\right),\qquad \eta,h\in k((t^{1/e})),
\]
and \(\calT\in R\) with \(\calT(z)\ne0\), the logarithmic contact order  of $z$ with respect to $\calT$ is defined to be
\[
  \mu_{z,\calT}=
  \ord_t\left(\frac{\dd(\calT(z))}{\calT(z)}-\frac{\dd z}{z}\right).
\]
\end{definition}
Let $u$ be a differential indeterminate. We also need to introduce the Bell differential polynomials $B_j(u)$. 
Set $B_0(u)=1$. For $j\geq 0$, set
$$
    B_{j+1}(u)=\delta(B_j(u))+uB_j(u).
$$
Set $C_j(u)=B_j(u)-u^j$ for each $j\geq 0$. Note that $C_0=C_1=0$.
\begin{lemma}
  \label{LM:Bellformula}
Suppose that $\calL=a_n\partial^n+a_{n-1}\partial^{n-1}+\dots+a_0\in R$.
Then 
$$
   \frac{\calL(z)}{z}=\sum_{i=0}^n a_i B_i(\gamma)
$$
where $\gamma=\frac{\delta(z)}{z}$.
\end{lemma}
\begin{proof}
  It suffices to show that $\frac{\delta^j(z)}{z}=B_j(\gamma)$ for all $0\leq j \leq n$.
  If $j=0$ then $B_0(\gamma)=1=\delta^0(z)/z$. Suppose that $\ell\geq 0$ and 
  the conclusion holds for $j=\ell$. Then 
  $$
    \frac{\delta^{\ell+1}(z)}{z}=\delta\left(\frac{\delta^\ell(z)}{z}\right)
    +\frac{\delta(z)}{z}\frac{\delta^\ell(z)}{z}=\delta(B_\ell(\gamma))+\gamma B_\ell(\gamma)=B_{\ell+1}(\gamma).
  $$
\end{proof}
\begin{lemma}
  \label{LM:orderofBell1}
  Suppose that $w\in k((t^{1/e}))$ with $\ord_t(w)<1$. Then for $j\geq 0$
  $$
     \ord_t(C_j(w))>j\ord_t(w).
  $$
\end{lemma}
\begin{proof}
  The conclusion is obvious when $j=0$ and $j=1$. Assume that $\ell\geq 1$ and 
  the conclusion holds for $j=\ell$. Then 
  \begin{align*}
     C_{\ell+1}(w)&=B_{\ell+1}(w)-w^{\ell+1}=\delta(B_\ell(w))+wB_\ell(w)-w^{\ell+1}\\
     &=\delta(C_\ell(w))+\delta(w^\ell)+wC_\ell(w).
  \end{align*}
  By the induction hypothesis, $\ord_t(C_\ell(w))>\ell\ord_t(w)$. Hence 
  $$
    \ord_t(\delta(C_\ell(w)))\geq \ord_t(C_\ell(w))+1>(\ell+1)\ord_t(w)
  $$
  and 
  $$
    \ord_t(w C_\ell(w))=\ord_t(w)+\ord_t(C_\ell(w))>(\ell+1)\ord_t(w).
  $$
  The lemma then follows from the fact that  
  $$
     \ord_t(\delta(w^\ell))\geq \ell\ord_t(w)+1>(\ell+1)\ord_t(w).
  $$
\end{proof}
\begin{lemma}
  \label{LM:orderofBell2}
  Suppose that $w_1,w_2\in k((t^{1/e}))$. Suppose that 
  $\ord_t(w_1)<\ord_t(w_2)<\infty$ and $\ord_t(w_1)<0$. Then for $j\geq 0$
  $$
     \ord_t(C_j(w_1+w_2)-C_j(w_1))> (j-1)\ord_t(w_1)+\ord_t(w_2).
  $$
\end{lemma}
\begin{proof}
We will prove by induction.When $j=0$ or $j=1$, $C_j(u)=0$ and the conclusion is obvious.
Assume that $\ell>0$ and the conclusion holds for $j=\ell$. Then
$$
   \ord_t(C_{\ell}(w_1+w_2)-C_{\ell}(w_1))> (\ell-1)\ord_t(w_1)+\ord_t(w_2).
$$
Direct calculation implies that
\[C_{\ell+1}(u) = \dd(u^{\ell})+\dd(C_{\ell}(u))+uC_{\ell}(u).\]
Thus we have the difference

\begin{align*}
  C_{\ell+1}(w_1+w_2) - C_{\ell+1}(w_1)
  &= \underbrace{\dd \left((w_1+w_2)^{\ell}-w_1^{\ell}\right)}_{A_1}
   + \underbrace{\dd \left(C_\ell(w_1+w_2) - C_\ell(w_1)\right)}_{A_2} \\
  &\quad + \underbrace{(w_1+w_2)\left(C_\ell(w_1+w_2)-C_\ell(w_1)\right)}_{A_3}
   + \underbrace{ w_2 \, C_\ell(w_1)}_{A_4}.
\end{align*}

It suffices to show that $\ord_t(A_i)>\ell\ord_t(w_1)+\ord_t(w_2)$ for all $1\leq i \leq 4$.
By Lemma~\ref{LM:order-polynomialdifference} 
$$
\ord_t(A_1)\geq (\ell-1)\ord_t(w_1)+\ord_t(w_2)+1>\ell\ord_t(w_1)+\ord_t(w_2),
$$
 Since $\dd$ raises the order of every nonzero term at least by 1, by the inductive hypothesis,  
$$\ord_t(A_2)>(\ell-1)\ord_t(w_1) + \ord_t(w_2)+1>\ell\ord_t(w_1)+\ord_t(w_2).$$
Since $\ord_t(w_1+w_2) = \ord_t(w_1)$, by the induction hypothesis,
$$
\ord_t(A_3)>\ord_t(w_1) + (\ell-1)\ord_t(w_1) + \ord_t(w_2) = \ell\ord_t(w_1) + \ord_t(w_2).
$$
By Lemma~\ref{LM:orderofBell1}, 
$$
 \ord_t(A_4)=\ord_t(w_2)+\ord_t(C_\ell(w_1))>\ell\ord_t(w_1)+\ord_t(w_2).
$$
\end{proof}

\begin{lemma}
  \label{LM:minimalorderinL2}
  Suppose that $\calL_1=P_1(\partial)+Q_1(x)$ and $\calL_2=P_2(\partial)+Q_2(x)$
  are two generalized Airy operators of type $(n,m)$ with $nm\neq 0$. Assume that 
  $z=\exp^{\int \eta dx} h$ with $\eta, h\in k((t^{1/e}))$ satisfying that
  $\calL_1(z)=0$ and $P_1(\eta)+Q_1(x)=0$. Let $\calT\in R$ satisfy $\calT(z)/z \notin k$, and set $\alpha=\frac{\delta(z)}{z}, \beta=\frac{\delta(\calT(z))}{\calT(z)}$. 
  Then the following assertions hold
  \begin{enumerate}
  \item [(1)] $\ord_t(P_2(\beta)-P_2(\alpha))=(n-1)\ord_t(\alpha)+\ord_t(\beta-\alpha).$
    \item [(2)]
    $\frac{\calL_2(\calT(z))}{\calT(z)}=(Q_2-Q_1)
 +(P_2(\alpha)-P_1(\alpha))
 +\bigl(P_2(\beta)-P_2(\alpha)\bigr)+M
  $
  \item [(3)]$\ord_t(M)>
     \min
     \{\ord_t(P_2(\alpha)-P_1(\alpha)), \ord_t(Q_2-Q_1),
     \ord_t(P_2(\beta)-P_2(\alpha))\}.
  $
  \end{enumerate}
\end{lemma}
\begin{proof}
Set $\Delta=\beta-\alpha$.
Since $\Delta=\delta(\calT(z)/z)/(\calT(z)/z)$ and $\calT(z)/z\notin k$,  Lemma~\ref{LM:logderivative} gives
$
  \ord_t(\Delta)\geq 1.
$
On the other hand,
\[
  \alpha=\eta+\frac{\delta (h)}{h}.
\]
By Proposition~\ref{PROP:Puiseuxsolutions} and 
Lemma~\ref{LM:logderivative}, $\ord_t(\eta)=-m/n<0$ and
$\ord_t(\delta(h)/h)\geq 1$. Hence we have
$
  \ord_t(\alpha)=-\frac{m}{n}<0.
$
Consequently,
$$
  \ord_t(\Delta)>\ord_t(\alpha).
$$
(1) Because $P_2$ is monic of degree $n$ and $\beta=\alpha+\Delta$, by Lemma~\ref{LM:order-polynomialdifference},
$$
   P_2(\beta)-P_2(\alpha)=P_2'(\alpha)\Delta+\text{higher order terms}.
$$
Therefore,
\[
  \ord_t(P_2(\beta)-P_2(\alpha))
  =\ord_t(P_2'(\alpha)\Delta)=(n-1)\ord_t(\alpha)+\ord_t(\Delta).
\]

(2) Write
$
  P_i(Y)=\sum_{j=0}^n p_{i,j}Y^j, i=1,2.
$
By Lemma~\ref{LM:Bellformula}, we have 
\[
  0=\frac{\calL_1(z)}{z}
   =Q_1+\sum_{j=0}^n p_{1,j}B_j(\alpha)=Q_1+P_1(\alpha)+\sum_{j=0}^n p_{1,j}C_j(\alpha)
\]
and
\[
  \frac{\calL_2(\calT(z))}{\calT(z)}
   =Q_2+\sum_{j=0}^n p_{2,j}B_j(\beta)=Q_2+P_2(\beta)+\sum_{j=0}^n p_{2,j}C_j(\beta).
\]
Subtracting the first identity from the second and using
$\beta=\alpha+\Delta$, we obtain
\[
\begin{aligned}
  \frac{\calL_2(\calT(z))}{\calT(z)}
  ={}&(Q_2-Q_1)+(P_2(\alpha)-P_1(\alpha))
     +\bigl(P_2(\beta)-P_2(\alpha)\bigr)\\
   &+\sum_{j=0}^n(p_{2,j}-p_{1,j})C_j(\alpha)
     +\sum_{j=0}^n p_{2,j}
       \bigl(C_j(\beta)-C_j(\alpha)\bigr).
\end{aligned}
\]
Thus we may take
\[
  M=M_1+M_2,
\]
where
\[
  M_1=\sum_{j=0}^n(p_{2,j}-p_{1,j})C_j(\alpha)\,\,\text{and}\,\,
  M_2=\sum_{j=0}^n p_{2,j}
       \bigl(C_j(\beta)-C_j(\alpha)\bigr).
\]
This proves (2).

(3) By Lemma~\ref{LM:orderofBell2}, for every nonzero summand of $M_2$,
\[
\begin{aligned}
  \ord_t\bigl(C_j(\beta)-C_j(\alpha)\bigr)
  &>(j-1)\ord_t(\alpha)+\ord_t(\Delta)\geq (n-1)\ord_t(\alpha)+\ord_t(\Delta),
\end{aligned}
\]
because $j\leq n$ and $\ord_t(\alpha)<0$. Together with (1), this yields
\begin{equation}
  \label{EQ:M2}
  \ord_t(M_2)>
  \ord_t\bigl(P_2(\beta)-P_2(\alpha)\bigr)
\end{equation}

Suppose that $P_1=P_2$. Then $M_1=0$. Hence
$$
  \ord_t(M)=\ord_t(M_2)>\ord_t\bigl(P_2(\beta)-P_2(\alpha)\bigr)
$$
and the desired inequality follows. 

Suppose that $P_1\neq P_2$. Because
$\ord_t(\alpha)<0$,
$$
  \ord_t(P_2(\alpha)-P_1(\alpha))=r\ord_t(\alpha).
$$
On the other hand, by Lemma~\ref{LM:orderofBell1}, for every nonzero summand of $M_1$, 
\[
\begin{aligned}
  \ord_t\bigl((p_{2,j}-p_{1,j})C_j(\alpha)\bigr)
  >j\ord_t(\alpha)
  \geq r\ord_t(\alpha)
   =\ord_t(P_2(\alpha)-P_1(\alpha)),
\end{aligned}
\]
because
$\ord_t(\alpha)<0$ and $j\leq r$. Hence
\begin{equation}
  \label{EQ:M1}
  \ord_t(M_1)> \ord_t(P_2(\alpha)-P_1(\alpha)).
\end{equation}

From (\ref{EQ:M1}) and (\ref{EQ:M2}), and the ultrametric inequality,
\[
\begin{aligned}
  \ord_t(M)
  &\geq \min\{\ord_t(M_1),\ord_t(M_2)\}\\
  &>\min\Bigl\{
      \ord_t\bigl((P_2-P_1)(\alpha)\bigr),
      \ord_t\bigl(P_2(\beta)-P_2(\alpha)\bigr)
     \Bigr\}.
\end{aligned}
\]
Since adjoining the additional number $\ord_t(Q_2-Q_1)$ can only decrease
this minimum, we conclude that
\[
\begin{aligned}
  \ord_t(M)>
  \min\bigl\{&\ord_t\bigl((P_2-P_1)(\alpha)\bigr),
  \ord_t(Q_2-Q_1),
  \ord_t\bigl(P_2(\beta)-P_2(\alpha)\bigr)\bigr\}.
\end{aligned}
\]
This proves the case $P_1\neq P_2$. Therefore (3) holds. 
\end{proof}

In the following, by convention, we set $\deg(0)=-\infty$.
\begin{theorem}
  \label{THM:degreeobstruction}
Let
$
  \calL_i=P_i(\partial)+Q_i(x), i=1,2 
$
be two generalized Airy operators of type $(n,m)$ with $nm\neq 0$.
Let $z=\exp\left(\int \eta\,dx\right)h$, where $h\neq 0$ and
$\eta,h\in k((t^{1/e}))$, satisfy
$
  P_1(\eta)+Q_1(x)=0
$
and $\calL_1(z)=0$.
Let $\calT\in k(x)[\partial]$ satisfy
$
  \frac{\calT(z)}{z}\notin k
$
and $\calL_2(\calT(z))=0.$ Assume further that $\calL_1\neq \calL_2$. 
Denote $r=\deg(P_1-P_2)$ and $s=\deg(Q_1-Q_2)$. Then at least one of 
\[
  ns=rm,\qquad n(s+\mu_{z,\calT})=(n-1)m,\qquad
  n\mu_{z,\calT}=(n-1-r)m
\]
holds. Moreover, if $P_1=P_2$ then $n(s+\mu_{z,\calT})=(n-1)m$;
  if $Q_1=Q_2$ then $n\mu_{z,\calT}=(n-1-r)m$.
\end{theorem}
\begin{proof}
  Set $\alpha=\frac{\delta(z)}{z},\beta=\frac{\delta(\calT(z))}{\calT(z)}$ and 
  $\Delta=\beta-\alpha$. Then 
  $$
     \ord_t(\alpha)=\ord_t(\eta)=-\frac{m}{n}\,\,\text{and}\,\,\ord_t(\Delta)=\mu_{z,\calT}.
  $$
  We first determine the orders of $Q_2-Q_1, P_2(\alpha)-P_1(\alpha)$
  and $P_2(\beta)-P_2(\alpha)$. We have 
  $$\ord_t(Q_2-Q_1)=s\ord_t(x)=-s$$
and 
$$
  \ord_t(P_2(\alpha)-P_1(\alpha))=r\ord_t(\alpha)=-\frac{rm}{n}.
$$
By the equality (1) of Lemma~\ref{LM:minimalorderinL2}, 
$$
   \ord_t(P_2(\beta)-P_2(\alpha))=(n-1)\ord_t(\alpha)+\ord_t(\Delta)=-\frac{(n-1)m}{n}+\mu_{z,\calT}.
$$
Since $\calL_2(\calT(z))=0$, 
the inequality (3) of Lemma~\ref{LM:minimalorderinL2} implies that at least two of 
$$
   \ord_t(Q_2-Q_1), \ord_t(P_2(\alpha)-P_1(\alpha)), \ord_t(P_2(\beta)-P_2(\alpha))
$$
equal. In other words, at least one of 
$$
   ns=mr, \,\,n(s+\mu_{z,\calT})=(n-1)m,\,\, n\mu_{z,\calT}=(n-1-r)m
$$
holds. If $P_2=P_1$ then $Q_1\neq Q_2$ and we must have 
$\ord_t(Q_2-Q_1)=\ord_t(P_2(\beta)-P_2(\alpha))$. This implies that $n(s+\mu_{z,\calT})=(n-1)m$. Similarly, if $Q_1=Q_2$ then $P_1\neq P_2$ and we must have 
$(n-1-r)m=n\mu_{z,\calT}$.
\end{proof}

\section{Main results}
For $\calL=a_n\pa^n+a_{n-1}\pa^{n-1}+\dots+a_0 \in R$, 
the Newton polygon at $t=0$ is the lower convex hull of the points
\[
  (j,\ord_t(a_j))
\]
for which $a_j\ne0$.
A slope of the Newton polygon is a rational number $r$ such that there exists a line of slope $r$ 
which intersects the Newton polygon at two distinct points such that all other points lie on or above this line.
\begin{lemma}
  \label{LM:non-resonant-image-gap}
Let $z=\exp^{\int \eta dx}h$ with $\eta,h\in k((t^{1/e}))$ being Puiseux series. Suppose that
 $\ord_t(\delta(z)/z) < 0$.
Let $\calT \in k[x][\pa]\setminus k$. Suppose that the Newton polygon
of $\calT$ does not have a slope equal to $-\ord_t(\delta(z)/z)$.
Then,
\[ \ord_t\left(\frac{\calT(z)}{z}\right) < 0. \]
In particular, $\ord_t\left(\frac{\calT(z)}{z}\right) \ne 0$.
\end{lemma}
\begin{proof}
Write $\calT = \sum_{j=0}^r p_j(x)\pa^j$, where $r = \ord(\calT)$, and $p_j(x)\in k[x]$ has degree $d_j$.
Denote $\alpha = \frac{\delta(z)}{z}$. 
By Lemma~\ref{LM:Bellformula}, we have
\[ \frac{\calT(z)}{z} = \sum_j^r p_j(x) \left(\alpha^j + C_j(\alpha)\right). \]
where every nonzero term of the correction $C_j(\alpha)$ has order strictly larger than 
$\ord_t(\alpha^j)=j\ord_t(\alpha)$. Therefore, 
\[
  \ord_t(p_j(x)(\alpha^j+C_j(\alpha))) = \ord_t(p_j(x) \alpha^j) = -d_j + j\ord_t(\alpha). 
\]
Set $\ell_j=-d_j + j\ord_t(\alpha).$
Notice that since $\ord_t(\alpha) < 0$, we have $\ell_j \le 0$ for all $j \ge 0$. 
Furthermore, $\ell_j < 0$ strictly whenever $j \ne 0$, and $\ell_0=-d_0 < 0$ when $d_0 \ne 0$.

If the minimum of these $\ell_j$ were achieved at two distinct points $j_1$, and $j_2$, we would have the equality
\[ -d_{j_1}+j_1\ord_t(\alpha) = -d_{j_2}+j_2\ord_t(\alpha). \]
Geometrically, this implies that the segment connecting $(j_1, -d_{j_1})$ and $(j_2, -d_{j_2})$ lies on a 
boundary line of the Newton polygon of $\calT$ corresponding to the slope $-\ord_t(\alpha)$, 
which contradicts the assumption that $-\ord_t(\alpha)$ is not a slope of the Newton polygon of $\calT$.
This guarantees that the minimum order is achieved at a unique term, thus the order of $\frac{\calT(z)}{z}$ equals to this value.

On the other hand, since $\calT \notin k$, at least one value of $\ell_j$ is negative and then the minimum is negative. Thus, we conclude
$\ord_t\left(\frac{\calT(z)}{z}\right)< 0. $
This completes the proof.
\end{proof}

\begin{corollary}
  \label{COR:non-resonant-image-gap}
  Let $z=\exp^{\int \eta dx}h$ with $\eta,h\in k((t^{1/e}))$ being Puiseux series. 
  Suppose that $r=\ord_t\left(\frac{\dd (z)}{z}\right)<0$. For any $\calT \in k[x][\pa]$ of order less than the denominator of $r$, 
  if the quotient $\frac{\calT(z)}{z} \notin k$, then
  \[ \ord_t\left(\frac{\calT(z)}{z}\right) < 0. \]
\end{corollary}
\begin{proof}
Every slope in the Newton polygon of $\calT$ is a rational number whose denominator in reduced form is not greater than $\ord(\calT)$.
Because the denominator of $r$ is greater than $\ord(\calT)$, 
$-r$ cannot appear as a slope of $\calT$. The result then directly follows from Lemma~\ref{LM:non-resonant-image-gap}.
\end{proof}

\begin{lemma}
  \label{LM:coefficientsoftransformer}
  Let $\calL_1$ and $\calL_2$ be two generalized Airy operators of type $(n,m)$ with $nm\neq 0$. If 
  $\calL_1$ and $\calL_2$ are equivalent by $\calT\in R$ with $\ord(\calT)<n$ then $\calT\in k[x][\pa]$.  
\end{lemma}
\begin{proof}
Suppose that $\calL\in k[x][\pa]$ and $\calL$ is monic. It is well-known that for each $c\in k$ if $f(x)\in k((x-c))$ is a solution of $\calL(y)=0$ 
then $f(x)\in k[[x-c]]$, moreover, $\calL(y)=0$ has a basis $f_1(x),\dots,f_n(x)$ of solutions in $k[[x-c]]$ of the form
$$
    f_i(x)=(x-c)^{i-1} + \text{higher order terms}, \quad 1 \leq i \leq n.
$$
Assume that $\calT=a_{n-1}\pa^{n-1}+\dots+a_0$. Then 
$$
  (\calT(f_1),\dots,\calT(f_n)) =(a_0,\dots,a_{n-1})\underbrace{\begin{pmatrix} f_1 & f_2 & \dots & f_n \\
    \delta(f_1) & \delta(f_2) & \dots & \delta(f_n) \\
    \vdots & \vdots & \ddots & \vdots \\
    \delta^{n-1}(f_1) & \delta^{n-1}(f_2) & \dots & \delta^{n-1}(f_n) 
  \end{pmatrix} }_{W_f(x)}.
$$
It is easy to see $\det(W_f(c))\neq 0$. Hence $W_f(x)^{-1}\in \GL_n(k[[x-c]])$. On the other hand, 
 $\calT(f_i)$ is a solution of $\calL_2(y)=0$, we have $\calT(f_i)\in k[[x-c]]$. Therefore, 
 $$(a_0,\dots,a_{n-1}) = (\calT(f_1),\dots,\calT(f_n)) W_f(x)^{-1} \in k[[x-c]]^n.$$
Since $c$ is arbitrary, we have $a_i\in k[x]$ for all $0\leq i \leq n-1$.
\end{proof}

\begin{theorem}
  \label{THM:Katzproblem}
  Let $\calL_1$ and $\calL_2$ be two generalized Airy operators of type $(n, m)$ with $nm\neq 0$ 
  and $\gcd(n,m)=1$. If $\calL_1$ and $\calL_2$ are equivalent over $k(x)$. Then $\calL_1 = \calL_2$.
\end{theorem}
\begin{proof}
  Write $\calL_1=P_1(\partial)+Q_1(x)$. By Theorem~\ref{THM:basisofsolutions},
  the solution space of $\calL_1(y)=0$ has a basis 
  $$
    z_1=\exp^{\int \eta_1 dx }h_1, \dots, z_n=\exp^{\int \eta_n dx} h_n
  $$
  where $\eta_i ,h_i\in k((t^{1/(2n)}))$ and $P_1(\eta_i)+Q_1(x)=0$.
  Suppose that $\calL_1$ and $\calL_2$ are equivalent over $k(x)$ by $\calT$. Then 
  we may choose $\calT$ to be of order less than $n$. By Lemma~\ref{LM:coefficientsoftransformer}, 
  $\calT\in k[x][\pa]$. Note that $\calT$ induces a linear isomorphism from 
  the solution space of $\calL_1(y)=0$ to that of $\calL_2(y)=0$. In particular, 
  $\calT(z_i)$ is a solution of $\calL_2(y)=0$ for all $1\leq i \leq n$ and $\calT(z_i)\neq 0$.

  Assume that $\calT(z_i)/z_i\in k$ for all $1\leq i \leq n$. Then $\calL_1$ and $\calL_2$
  have the same solution space and therefore $\calL_1=\calL_2$. The proof is complete. 
  
  Suppose that there exists $z_\rho$ with $1\leq \rho \leq n$ such that $\calT(z_\rho)/z_\rho \notin k$. Denote 
  $$
   \alpha=\frac{\delta(z_\rho)}{z_\rho}=\eta_\rho+\frac{\delta(h_\rho)}{h_\rho}.
  $$
  Hence $\ord_t(\alpha)=\ord_t(\eta_\rho)=-m/n<0$.
  Since $\gcd(n,m)=1$, the denominator of $\ord_t(\alpha)$ is $n$. By Corollary~\ref{COR:non-resonant-image-gap},
  $\ord_t(\calT(z_\rho)/z_\rho)\neq 0$. Due to Lemma~\ref{LM:logderivative}, we have $\mu_{z_\rho,\calT} = 1$.
We claim that none of the equalities
\begin{equation}
\label{EQ:threeequalities}
ns=rm,\qquad n(s+1)=(n-1)m,\qquad
n=(n-1-r)m
\end{equation}
can hold.
Suppose first that $ns=rm$. Since $\gcd(n,m)=1$, it follows that $n\mid r$. On the other hand, both $\calL_1$ and $\calL_2$ are monic, so $r<n$. Hence $r=0$, which implies $s=0$. This contradicts the assumption that $Q_1(0)=Q_2(0)$, from which $s\neq 0$.
Next, suppose that
$
n(s+1)=(n-1)m.
$
Again, since $\gcd(n,m)=1$, we have $n\mid (n-1)$. Hence $n=1$. Substituting this into the above equality yields $s+1=0$, contradicting the fact that $s>0$.
Finally, suppose that
$
n=(n-1-r)m.
$
Since $\gcd(n,m)=1$, it follows that
$
n\mid (n-1-r).
$
As $0\le n-1-r<n$, we must have $n-1-r=0$. Consequently,
$
n=(n-1-r)m=0,
$
which is impossible. This proves the claim. 
By Theorem~\ref{THM:degreeobstruction}, if $\calL_1 \neq \calL_2$, then  one of the three equalities in (\ref{EQ:threeequalities}) must hold.
We have shown that each of them is impossible. Hence $\calL_1=\calL_2$.
\end{proof}

For $\calL=a_n\pa^n+\dots+a_0\in R$, the adjoint of $\calL$ is defined to be 
$$
    \calL^*=\sum_{i=0}^n (-1)^{n-i} \pa^i a_i.
$$
\begin{definition}
  Let $\calL\in R$ be a generalized Airy operator. We say $\calL$ is self-dual if $\calL$ is equivalent over $k(x)$ to $\calL^*$.
\end{definition}
One sees that the adjoint of a generalized Airy operator is still a generalized Airy operator of the same type.
\begin{theorem}
  \label{THM:self-dual}
 Let $\calL=P(\partial)+Q(x)$ be a generalized Airy operator of type $(n,m)$ and $nm\neq 0$. Suppose that $\gcd(n,m)=1$. Then 
 $\calL$ is self-dual if and only if $n$ is even and $P(Y)=S(Y^2)$ for some $S\in k[Y]$.
\end{theorem}
\begin{proof}
  Suppose that $\calL$ is self-dual, i.e. $\calL$ is equivalent over $k(x)$ to $\calL^*$. By Theorem~\ref{THM:Katzproblem},
  $\calL=\calL^*$. If $n$ is odd then $Q(x)=(-1)^nQ(x)$ and so $Q(x)=0$, a contradiction. Hence $n$ is even. Write 
  $$
     P(Y)=c_{2\ell}Y^{2\ell}+c_{2\ell-1}Y^{2\ell-1}+\dots+c_0.
  $$
  Then $P(\pa)^*=c_n\pa^n-c_{n-1}\pa^{n-1}+\dots+c_0$. As $\calL=\calL^*$, $P(\pa)^*=P(\pa)$ and thus $c_{2i+1}=0$ for all $0\leq i \leq \ell-1$.
  Therefore $P(Y)=S(Y^2)$ where $S=\sum_{i=0}^\ell c_{2i}Y$.

  Assume that $n$ is even and $P(Y)=S(Y^2)$ for some $S\in k[Y]$. Direct calculation implies that $\calL=\calL^*$ and so $\calL$ is self-dual.
\end{proof}
\bibliographystyle{alpha}
\bibliography{bibdata}
\end{document}